\documentclass[11pt]{amsart}
\usepackage{latexsym,amssymb}

\usepackage{graphicx} 
\usepackage{amsmath,amsthm}
\usepackage{amsfonts}
\usepackage{mathtools}
\usepackage{ amssymb }
\usepackage{ stmaryrd }
\usepackage{tikz-cd}
\usepackage[shortlabels]{enumitem}
\usepackage{ upgreek }
\usepackage{yfonts}
\usepackage{dsfont}
\usepackage{url}
\usepackage{graphicx}
\usepackage{hyperref}

\usepackage{afterpage}
\usepackage{ amssymb }
\usepackage{yfonts}
\usepackage[colorinlistoftodos]{todonotes}
\usepackage{enumitem}
\usepackage{fancyhdr}
\usepackage{orcidlink}
\input amssym.def \input amssym

\newtheorem{te}{Theorem}[section]

\newtheorem{de}[te]{Definition}
\newtheorem{lm}[te]{Lemma}

\newtheorem{pp}[te]{Proposition}
\newtheorem{co}[te]{Corollary}
\newtheorem{ex}[te]{Example}
\newtheorem{pb}[te]{Problem}
\newtheorem{re}[te]{Remark}

\newtheorem{alphthm}{Theorem}

\newcommand{\N}{\mathbb{N}}

\title{Remarks on ultrafilter exponentiation}

\author{Lorenzo Luperi Baglini}
\address{$^*$ Dipartimento di Matematica, Universit\`a  di Milano, Via Saldini 50, 20133 Milano, Italy}
\author{Boris \v Sobot}
\address{$^\dagger{}^\dagger$ Faculty of Sciences, University of Novi Sad, Trg Dositeja Obradovi\' ca 4, 21000 Novi Sad, Serbia}


\begin{document}

\begin{abstract}Recently, two exponentiations of ultrafilters have been introduced in \cite{LB}. We rephrase them in nonstandard terms, and use the nonstandard formulations to characterize some of their properties. We also study their relations with divisibility by introducing factor-picking functions and highly exponential ultrafilters.
\end{abstract}

\maketitle

\section{Introduction}

Combinatorial Ramsey theory is the study of the so-called partition regular configurations and equations on $\N=\{1,2,\ldots\}$. 

\begin{de} Let $f_{1}\left(x_{1},\ldots,x_{n}\right),\ldots,f_{m}\left(x_{1},\ldots,x_{n}\right):\N^{n}\rightarrow \N$. We say that the configuration $\{f_{1}\left(x_{1},\ldots,x_{n}\right),\ldots,f_{m}\left(x_{1},\ldots,x_{n}\right)\}$ is partition regular on $\N$ (PR from now on) if for all finite partitions $\N=A_{1}\cup\ldots\cup A_{k}$ there are $i\leq k$ and $x_{1},\ldots,x_{n}\in\N$ such that for all $j\leq m$ $f_{j}\left(x_{1},\ldots,x_{n}\right)\in A_{i}$.\end{de}

This area of combinatorics has been active for more than a century, with most results focusing on (finite and infinite) polynomial configurations. In recent years, there has been a growing interest on exponential configurations. Ultrafilters have been a major tool for this line of research; for an account on these developments we refer to \cite{DNR1, DNR2}, where Di Nasso and Ragosta have provided a simple and elegant proof of the partition regularity of $\{x,y,x^{y}\}$, as well as of several finite and infinite generalizations, with a clever use of special kinds of ultrafilters. 

These results have led the first named author of the present paper to introduce, in \cite{LB}, two notions of exponentiation between ultrafilters, mimicking the usual extension of binary operations from $\N^{2}$ to $\beta\N^{2}$:

\begin{de}\label{def:exp} Given $p,q\in\beta\mathbb{N}$, we let
\begin{eqnarray*}
E_1({p},{q}) &=& \{A\subseteq\mathbb{N}:\{n\in\mathbb{N}:\{m\in\mathbb{N}:n^m\in A\}\in{q}\}\in{p}\};\\
E_2({p},{q}) &=& \{A\subseteq\mathbb{N}:\{n\in\mathbb{N}:\{m\in\mathbb{N}:m^n\in A\}\in{q}\}\in{p}\}.
\end{eqnarray*}\end{de}

The extension of the exponentiation to ultrafilters fits a recent line of research based on the extensions of relations from $\N^{n}$ to $(\beta \N)^{n}$; we refer to \cite{gasparroluperi26} for general results and examples. Amongs these extensions, a relevant place is occupied by divisibility and congruences. These notions have been thoroughly studied by the second named author of this paper in a recent series of work, see \cite{So1, So2, So5, So8}. Most of this study has been based on a well-developed technique in the area, based on the identification of ultrafilters with their set of nonstandard generators; see e.g. \cite{DN,luperi2019nonstandard,L,P1,P2}. The basic idea of this approach is that several combinatorial properties of ultrafilters, and many relationships between different ultrafilters, can be better studied by translating them into properties of their generators. 

In this paper, we adopt this same nonstandard point of view to study certain properties of the ultrafilters exponentiations, and we start to inquire their relationships with divisibility. To this end, we introduce a class of standard functions well suited for our task, called factor-picking, as well as a family of extremal ultrafilters with respect to factor-picking functions, defined as follows.

\begin{de}\label{def:fp} Let $f:\N\rightarrow\N$. We say that $f$ is a factor-picking function if
\begin{enumerate}
\item for all $n\in\N$ $f(n)$ divides $n$;
\item for all $n,m\in\N$ $f(n^{m})=f(n)$.
\end{enumerate}
Given a factor-picking $f$, we say that $e_{f}:\N\rightarrow \N$ is its exact exponent if for all $n\in\N$ 
\[e_{f}(n)=\begin{cases} 1, & \text{if} \ f(n)=1;\\
\max\{m\in\N\mid f(n)^{m}\ \text{divides} \ n\}, &\text{otherwise.}\end{cases}\]

We say that ${p}\in\beta\mathbb{N}\setminus\N$ is $f$-highly exponential if for all $g:\N\rightarrow \N$ 
\[\{n\in\N\mid g(f(n))<e_{f}(n)\}\in p.\]

Finally, we say that $p$ is highly exponential if it is $f$-highly exponential for all factor picking functions with $f(p)\neq 1$.
\end{de}

A particularly relevant example of factor-picking function is the radical function ${\rm rad}$, that maps $n$ to the product of its prime factors. Our first main result is the following theorem, relating highly exponential ultrafilters, ${\rm rad}$ and divisibility.

\begin{alphthm}\label{alphthm:A}
Let $p,q\in\beta\mathbb{N}\setminus\{1\}$. 
\begin{enumerate}[label=(\roman*)]
\item If $p$ is highly exponential and ${\rm rad}(q)$ divides ${\rm rad}(p)$, then $q$ divides $p$.
\item Assume that $rad(p)$ does not strictly divide itself and that, for all $q$ such that $rad(q)=rad(p)$, we have that $q$ divides $p$. Then $p$ is highly exponential.
\end{enumerate}
\end{alphthm}

Our second main result regards the equation $E_{1}(p,q)=E_{2}(p,q)$; identifying its solutions was left as an open question in \cite{LB}. When at least one between $p,q$ is principal, it is simple to show that $E_{1}(p,q)=E_{2}(p,q)$ holds true. However, when both $p,q$ are nonprincipal, this equality is no longer guaranteed: as shown in \cite[Section 2]{LB}, it holds for certain pairs $(p,q)$, and fails for others. In Section \ref{sec:twoexp} we prove that solutions to $E_{1}(p,q)=E_{2}(p,q)$ must satisfy strict restrictions.
 
\begin{alphthm}\label{alphthm:B}
Let ${p},{q}\in\beta\mathbb{N}$ be such that $E_1({p},{q})=E_2({q},{p})$. Let $f$ be a factor-picking function with $f(p)\neq 1$ and let $e_{f}$ be its exact exponent. Then 
\begin{enumerate}[label=(\roman*)]
\item $p$ is $f$-highly exponential;
\item $e_{f}({p})\odot{q}={q}\odot e_{f}({p})$.
\end{enumerate}
\end{alphthm}

The paper is organized as follows. In Section \ref{sec:nsa} we recall many basic facts about the nonstandard approach to ultrafilters that are useful for our purposes. In Section \ref{sec:heu} we focus on the properties of highly exponential ultrafilters, and we prove Theorem \ref{alphthm:A}. In Section \ref{sec:twoexp} we focus on the problem $E_{1}(p,q)=E_{2}(p,q)$, and we prove our Theorem \ref{alphthm:B}. In Section \ref{sec:genexp} we consider the most general exponentiation available on $\beta(\N^{2})$, and we generalize the problem $E_{1}(p,q)=E_{2}(p,q)$ to this setting proving some preliminary result towards its solution. Finally, in Section \ref{sec:conclusions} we collect some open problems generated by our results.

\section{Nonstandard generators}\label{sec:nsa}

We will work in a nonstandard extension $^{\ast}\N$ of ${\mathbb{N}}$, which we assume to be a ${2^{|\beta\mathbb{N}|}}^+$-saturated monster model; this ensures that internal sets in ${\mathbb{N}}$ are either finite or of (external) cardinality larger than that of $\beta\mathbb{N}$. We will follow the common practice of keeping the use of $\ast$ in front of symbols to the bare minimum.

The key concept we will use is that of nonstandard generator of an ultrafilter; the interested reader can find a thorough study of such generators in \cite{DN}.

\begin{de} Let $n\in\N$, let $p\in\beta(\N^{n})$. We say that $\alpha\in\,^{\ast}\N^{n}$ is a generator of $p$ if, for all $A\in p$, $\alpha\in\,^{\ast}A$. In this case we will write that $p=p_{\alpha}$ or, equivalently, that $\alpha\models p$. We let $\mu(p)$ denote the set of generators of $p$, and write $\alpha\sim\beta$ whenever $\alpha,\beta$ generate the same ultrafilter.\end{de}

Notice that $\alpha\sim\beta$ means, equivalently, that $\alpha$ and $\beta$ belong to the same standard sets; hence, whenever $P(x)$ is a standard property, we have that $P(\alpha)$ holds if and only if $\alpha\in\,^{\ast}\{n\in\N\mid P(n)\}$ if and only if $\beta\in\,^{\ast}\{n\in\N\mid P(n)\}$ if and only if $P(\beta)$ holds. We will use this fact throughout the paper without ever mentioning it again.

In this paper we will be mostly interested on ultrafilters over $\N$ and $\N^{2}$. In the latter case, a special role is played by the so-called tensor product of ultrafilters, whose definition we recall here (see \cite{luperi2019nonstandard} for a detailed study of tensors from the nonstandard perspective).

\begin{de} Let $p\in\beta(\N^{2})$, $q,r\in\beta\N$. 

\begin{enumerate}
\item[(a)] The product $q\times r$ is the filter generated by sets $A\times B$ for $A\in q$, $B\in r$.
\item[(b)] An ultrafilter $p$ is the tensor product of $q,r$ (notation: $p=q\otimes r$) if for all $A\subseteq\N^{2}$ we have that
\[A\in p\Leftrightarrow \{n\in\N\mid \{m\in\N\mid (n,m)\in A\}\in r\}\in q.\]
\end{enumerate}
A pair $(\alpha,\beta)\in\,^{\ast}\N$ will be called a tensor pair if it generates a tensor product.
\end{de}

The last nonstandard notions we need to recall are those of ``being arbitrarily larger'' and exact image.

\begin{de} Let $\alpha,\beta\in\,^{\ast}\N$. We say that $\alpha\ll\beta$ if for all $f:\N\rightarrow\N$ we have that $f(\alpha)<\beta$. We let ${\rm er}(\alpha):=\{f(\alpha)\mid f:\N\rightarrow\N\}\setminus\N$.\end{de}

We collect here some known properties of the above notions that we will use in the following.

\begin{lm}\label{lem:list} Let $p,q\in\beta\N$ $r\in\beta(\N^{2})$, let $\alpha,\beta,\gamma\in\,^{\ast}\N$, let $f,g:\N\rightarrow\N$. The following facts hold:
\begin{enumerate}[label=(\roman*)]
\item\label{T1} If $\alpha\sim \beta$ and $h:\N\rightarrow\N^k$, then $h(\alpha)\sim h(\beta)$.
\item\label{T2} If $f(\alpha)\sim\beta$, then there exists $\alpha^{\prime}\sim\alpha$ with $f(\alpha^{\prime})=\beta$.
\item\label{T3} If $f(\alpha)\sim \alpha$, then $f(\alpha)=\alpha$.
\item\label{T4} The following are equivalent.
\begin{enumerate}
\item $f(\alpha)\sim g(\beta)$.
\item There exists $\beta'\sim \beta$ such that $f(\alpha)= g(\beta')$.
\end{enumerate}
\item\label{T5} $(\alpha,\beta)$ is a tensor pair if and only if for all $f:\N\rightarrow \N$ we have $f(\beta)\in\N$ or $\alpha<f(\beta)$.
\item\label{T6} For all $n\in\N$, $(n,\alpha)$ and $(\alpha,n)$ are tensor pairs.
\item\label{T7} If $(\alpha,\beta)$ is a tensor pair, then $p_{(\alpha,\beta)}=p_{\alpha}\otimes p_{\beta}$.
\item\label{T8} If $r\supseteq p_{\alpha}\times p_{\beta}$, then for all $(\eta,\sigma)\in \mu(r)$ we have $\eta\sim\alpha,\sigma\sim\beta$.
\item\label{T9} If $r\supseteq p_{\alpha}\times p_{\beta}$, there exist $\alpha^{\prime}\sim\alpha$, $\beta^{\prime}\sim \beta$ such that $(\alpha,\beta^{\prime})\sim (\alpha^{\prime},\beta)\in\mu(r)$.
\item\label{T10} If $(\alpha,\beta)$ is a tensor pair, then so is $(f(\alpha),g(\beta))$.
\item\label{T11} If $(\alpha,\beta)$ is a tensor pair, with $\beta\in\,^{\ast}\N\setminus\N$, then $\alpha\ll\beta$. The converse does not hold.
\item\label{T12} $\alpha\ll \beta\gamma$ if and only if $\alpha\ll \beta$ or $\alpha\ll\gamma$.
\item\label{T13} Let $\gamma\in\,^{\ast}\N$. If $\alpha\leq\beta\ll\gamma\leq\delta$ then $\alpha\ll\delta$.
\end{enumerate}
    
\end{lm}

\begin{proof} \ref{T1}-\ref{T2}-\ref{T3} These are well known trivial facts, see e.g. \cite[Proposition 11.2.4-Theorem 11.2.7]{DN}.
		
		\ref{T4} This is a particular istance of \ref{T2}.

    \ref{T5} This is Puritz's Theorem, see \cite{P2}.

    \ref{T6} This is immediate from \ref{T5}.
    
    \ref{T7}-\ref{T8} These are immediate, see \cite[Proposition 4.6]{luperi2019nonstandard}.

    \ref{T9} This is a standard saturation argument, see e.g. \cite[Proposition 2.3.(viii)]{gasparroluperi26}.

    \ref{T10} This is a rephrasing of \cite[Theorem 4.15.(7)]{luperi2019nonstandard}

    \ref{T11} If $f:\N\rightarrow\N$ is such that $f(\alpha)>\beta$, the function $g:\N\rightarrow\N$ defined by $g(n)=\min \{m\in\N\mid f(m)>n\}$ is such that $g(\beta)\leq\alpha$, $g(\beta)\notin\N$, contradicting \ref{T5}. The converse does not hold because, e.g, if $(\alpha,\beta)$ and $(\beta,\gamma)$ are tensor, then $2^{\alpha}3^{\beta}\ll 2^{\beta}3^{\gamma}$ but the $2$-adic valuation of the latter is smaller than the former, contradicting \ref{T5}.

\ref{T12} Right to left is immediate. As for left to right, assume the opposite, that there are $f:\mathbb{N}\rightarrow\mathbb{N}$ and $g:\mathbb{N}\rightarrow\mathbb{N}$ such that $ f(\alpha)\geq \beta$ and $ g(\alpha)\geq \gamma$. If we define $h(n)=f(n)\cdot g(n)$, clearly $ h(\alpha)\geq \beta\gamma$, a contradiction.

\ref{T13}. Assume that $f(\alpha)\geq\delta$. Let $h:\N\rightarrow\N$ be the function mapping $n$ to $\max\{f(m)\mid m\leq n\}$. Then $h(\beta)\geq f(\alpha)\geq \gamma$, against $\beta\ll\gamma$. \end{proof}

\section{Highly exponential ultrafilters}\label{sec:heu}

In this section we work with factor-picking functions, and highly exponential ultrafilters, as introduced in Definition \ref{def:fp}. Amongst these, a special role is played by the principal basis and the radical functions.

\begin{de}
For $n\in{\mathbb{N}}\setminus\{1\}$, we let the principal exponent of $n$ (denoted by $PE(n)$) to be the largest $b\in{\mathbb{N}}$ such that there exists $a\in\N$ with $n=a^b$. Moreover, we let the principal basis of $n$ (denoted by $PB(n)$) to be the unique $a\in\N$ such that $n=a^{PE(n)}$. For $n=1$ we let $PB(1)=PE(1)=1$.

We also let the radical of $n$ (notation: ${\rm rad}(n)$) to be the product of the prime divisors of $n$.
\end{de}

Hence, every $n\in\N\setminus\{1\}$ can be written uniquely in the form $n=PB(n)^{PE(n)}$. We recall that, when $PE(n)\geq 2$, $n$ is said to be a perfect power; in particular, $n=PB(n)$ if and only if $n$ is not a perfect power. It is also immediate to notice that $PE(n)$ is the greatest common divisor of the exponents of its prime factorization. We also notice that the exact exponent of ${\rm rad}$ is the function $e_{\rm rad}$ mapping $n$ to the smallest exponent appearing in its prime factorization. 

Factor-picking functions are rather common. We list here a few more relevant examples.

\begin{ex} The following are all factor-picking functions:
\begin{enumerate}
\item for all prime $p$, the function that maps $n$ to $p$ if $p$ divides $n$, or to $1$ otherwise;
\item for all $m\in\N$, the function mapping $n$ to its $m$-th prime divisor, if $n$ has at least $m$ divisors, or to $1$ otherwise;
\item the function mapping $n$ to the product of its first $m$ prime divisors, if $n$ has at least $m$ divisors, or to $1$ otherwise;
\item the function mapping $n$ to the product of its last $m$ prime divisors, if $n$ has at least $m$ divisors, or to $1$ otherwise.
\end{enumerate}
\end{ex}

A trivial, yet useful, fact that we will use is that, if $f$ is a factor-picking function and $e_{f}$ is its exact exponent, by definition $e_{f}(a^{b})=b\cdot e_{f}(a)$ whenever $f(a)\neq 1$. As said in the introduction, factor-picking functions are related to $f$-highly exponential ultrafilters (see Definition \ref{def:fp}). Being $f$-highly exponential is more easily seen through generators: call $\alpha\in\,^{\ast}\N$ $f$-highly exponential if $f(\alpha)\ll e_{f}(\alpha)$. Then it can be routinely checked that $p$ is $f$-highly exponential iff there is an $f$-highly exponential $\alpha\in\mu(p)$ iff all $\alpha\in\mu(p)$ are $f$-highly exponential. When $f(\alpha)=1$, i.e. when $\alpha\in\,^{\ast}\{n\in\N\mid f(n)=1\}$, no power $\alpha^{\beta}$ can be $f$-highly exponential. On the contrary, when $f(\alpha)\neq 1$, powers of $\alpha$ can be $f$-highly exponential: for example, if $\alpha\ll\beta$ then $\alpha^{\beta}$ is $f$-highly exponential for all $f$ factor-picking such that $f(\alpha)\neq 1$. Actually, the following characterization, clearing the special role of $PB$ and ${\rm rad}$ amongst factor-picking functions, holds true.

\begin{pp}\label{maxhighexp}
For every $p\in\beta\N\setminus\N$, the following conditions are equivalent:

\begin{enumerate}[label=(\roman*)]
\item\label{E1} $p$ is highly exponential;

\item\label{E5} $p$ is $f$-highly exponential for some factor-picking function $f$ such that ${\rm rad}(\alpha)\mid f(\alpha)$ for some $\alpha\models p$;

\item\label{E2} $p$ is ${\rm rad}$-highly exponential;

\item\label{E3} $p$ is $PB$-highly exponential;

\item\label{E4} There exist $\alpha\ll\beta, \alpha>1$ such that $\alpha^{\beta}\models p$.
\end{enumerate}
\end{pp}

\begin{proof} 
\ref{E1}$\Rightarrow$\ref{E2}$\Rightarrow$\ref{E5} and \ref{E3}$\Rightarrow$\ref{E4} are obvious.

\ref{E2}$\Rightarrow$\ref{E3} Let $\alpha\in\mu(p)$ be such that ${\rm rad}(\alpha)\ll e_{{\rm rad}}(\alpha)$. For all $n\in\N$ ${\rm rad}(n)\leq PB(n)$ and $e_{{\rm rad}}(n)\geq PE(n)$. By transfer, ${\rm rad}(\alpha)\leq PB(\alpha)$ and $e_{\rm rad}(\alpha)\geq PE(\alpha)$, and we conclude by Lemma \ref{lem:list}.\ref{T13}.

\ref{E5}$\Rightarrow$\ref{E2} Let $f$ be factor-picking such that ${\rm rad}(\alpha)\mid f(\alpha)$ for some $\alpha\models p$, so $\{n\in\N\in{\rm rad}(n)\mid f(n)\}\in p$. Hence $\{n\in\N:e_{\rm rad}(n)\geq e_f(n)\}\in p$ and consequently $e_{\rm rad}(\alpha)\geq e_f(\alpha)$. Thus $f(\alpha)\ll e_{f}(\alpha)$ implies ${\rm rad}(\alpha)\ll e_{\rm rad}(\alpha)$.

\ref{E4}$\Rightarrow$\ref{E1} Let $f$ be a factor-picking function. $f(\alpha^{\beta})=f(\alpha)\ll\beta$ by hypothesis, whilst $e_{f}(\alpha^{\beta})=\beta\cdot e_{f}(\alpha)\geq \beta$. We conclude again by Lemma \ref{lem:list}.\ref{T13}.
\end{proof}

In general, being $f$-highly exponential forces restrictions on factorizations.

\begin{pp}\label{highexp}
Let $\alpha=\prod_{i\leq z}p_i^{k_i}$. Let $f$ be a factor-picking function and $e_{f}$ its exact exponent. If $f(\alpha)\ll e_{f}(\alpha)$ then, for all $i\leq z$ such that $p_{i}$ divides $f(\alpha)$, $p_i\ll k_i$.
\end{pp}

\begin{proof} For all $i\leq z$ such that $p_{i}$ divides $f(\alpha)$ we have $p_{i}\leq f(\alpha)$ and similarly $k_{i}\geq e_{f}(\alpha)$; thus, $p_{i}\leq f(\alpha)\ll e_{f}(\alpha)\leq k_{i}$, from which $p_{i}\ll k_{i}$ follows immediately.\end{proof}

\begin{ex}\label{exweakhe}
The reverse implication does not hold, in the following sense. Take $p_1,p_2$ infinite nonstandard primes, and $k_1,k_2\in\,^{\ast}\mathbb{N}$ such that $p_1\ll k_1<p_2\ll k_2$. If $\alpha=p_1^{k_{1}}p_{2}^{k_{2}}$, then ${\rm rad}(\alpha)=p_1p_2$ and $e_{\rm rad}(\alpha)=k_1$, so ${\rm rad}(\alpha)>e_{\rm rad}(\alpha)$.
\end{ex}

To better undestand highly exponential ultrafilters, we need a characterization of which pairs $(\alpha,\beta),(\alpha,\beta^{\prime})$ are such that $\alpha^{\beta}\sim\alpha^{\beta^{\prime}}$. This has been done in \cite[Lemma 22]{LMMRS}. Stated in our current language, with $A=\{\alpha\}$, \cite[Lemma 22]{LMMRS} says the following:

\begin{lm}\label{FundRemCut}
If $\beta=f(\alpha)$ for some $f:\N\rightarrow \N$, then $\alpha^{\beta}\sim \alpha^{\beta^{\prime}}$ if and only if $\beta=\beta^{\prime}$.

If $\beta\neq f(\alpha)$ for all $f:\N\rightarrow\N$, let $F_{\beta}=\{f:\N\rightarrow \N\mid f(\alpha)<\beta\}$ and let $I_{\beta}=\{\gamma\in\,^{\ast}\N\mid \forall f\in F_{\beta}\ \forall g\in F_{\beta}^{c}\ f(\alpha)<\gamma<g(\alpha)\}$. Then:
\begin{enumerate}[label=(\roman*)]
\item $\alpha^{\beta}\sim \alpha^{\beta^{\prime}}\Rightarrow F_{\beta}=F_{\beta^{\prime}}$;
\item the set $\{\beta^{\prime}\in\,^{\ast}\N\mid \alpha^{\beta}\sim \alpha^{\beta^{\prime}}\}$ is coinitial and cofinal in $I_{\beta}$.
\end{enumerate}
\end{lm}

In particular, when $\alpha\ll\beta$ we have that $\alpha^{\beta}\sim\alpha^{\beta^{\prime}}$ for arbitrarily large $\beta^{\prime}$. By Lemma \ref{lem:list}.\ref{T11}, all ultrafilters generated by some $\alpha^{\beta}$ with $(\alpha,\beta)$ tensor are highly exponential, but the reverse implication does not hold. However, highly exponential ultrafilters have a nice behaviour with respect to divisibility between ultrafilters. We recall the nonstandard definition, that is more suited for our present approach.

\begin{de}
Let $p,q\in\beta\N$. We say that:
\begin{enumerate}
\item[(a)] $p$ divides $q$ if there exist $\alpha\in\mu(p),\beta\in\mu(q)$ such that $\alpha|\beta$;
\item[(b)] $p$ strictly divides itself if there exist distinct $\alpha,\beta\in\mu(p)$ such that $\alpha|\beta$.
\end{enumerate}
\end{de}

A saturation argument easily shows that $p$ divides $q$ iff for every $\alpha\in\mu(p)$ there exists $\beta\in\mu(q)$ such that $\alpha|\beta$ iff for every $\beta\in\mu(q)$ there exists $\alpha\in\mu(p)$ such that $\alpha|\beta$, see \cite[Theorem 3.4]{So5}.

In standard terms, $p$ strictly dividing itself is equivalent to say that $p$ contains an antichain for the division; see \cite[Proposition 7]{LMMRS} for details\footnote{The reference talks about types and orders; to apply it here, just observe that an ultrafilter over $\N$ can be identified with a complete type in the complete language over $\N$, and that divisibility is an order.}. We make more precise the relationship between highly exponential ultrafilters and divisibility by proving our Theorem \ref{alphthm:A}

\begin{proof}[Proof of Theorem \ref{alphthm:A}] (i) Let ${p}$ be highly exponential and generated by $\gamma_{2}$, and let $q$ be generated by $\gamma_{1}=\prod_{i\leq z}p_i^{k_i}$, with ${\rm rad}(\gamma_{1})=\eta_{1} | \eta_{2} ={\rm rad}(\gamma_{2})$. Since an internal set bounded from above has a maximum, let $\mu_1=\max\{k_i:i\leq z\}$, so $\gamma_{1} | \eta_{1}^{\mu_{1}}$. By Lemma \ref{FundRemCut}, pick $\mu_{2}>\mu_{1}$ with $\eta_{2}^{\mu_{2}}\sim\gamma_{2}$. With this choice we have that $\eta_{1}^{\mu_{1}}|\eta_{2}^{\mu_{2}}$, hence $\gamma_{1}| \eta_{2}^{\mu_{2}}$, as desired.

(ii) By contrast, assume that $p$ is not highly exponential. Let $\alpha\models PB(p),\alpha^{\beta}\models p$ and let $f:\N\rightarrow \N$ be such that $\beta<f(\alpha)$. Let $h$ be the function mapping $n$ to the maximal exponent appearing in the prime factorization of $n$, and let $g=h\cdot f$. In particular, $\beta<g(\alpha)$. Let $q$ be such that ${\rm rad}(\alpha)^{g(\alpha)}\models q$. Then ${\rm rad}(p)={\rm rad}(q)$, hence by hypothesis some generator $\delta$ of $q$ divides $\alpha^{\beta}$. But generators of $q$ have the form ${\rm rad}(\gamma)^{g(\gamma)}$ for some $\gamma\sim \alpha$, and ${\rm rad}(\gamma)^{g(\gamma)}|\alpha^{\beta}$ entails that ${\rm rad}(\gamma)|\alpha^{\beta}$. As all primes in the factorization of ${\rm rad}(\gamma)$ have exponent $1$, this entails that ${\rm rad}(\gamma)|{\rm rad}(\alpha)$; since ${\rm rad}(p)$ does not strictly divide itself, this forces $\gamma=\alpha$. Hence ${\rm rad}(\alpha)^{g(\alpha)}|\alpha^{\beta}$, which is impossible as it would give $g(\alpha)=h(\alpha)\cdot f(\alpha)\leq h(\alpha)\cdot \beta$, hence $f(\alpha)\leq\beta$.\end{proof}

\begin{co}  Let $p\in\beta\N$. If $rad(p)$ does not strictly divide itself then $p$ is highly exponential if and only if, for all $q\in \beta\mathbb{N}\setminus\{1\}$, if $rad(q)=rad(p)$ then $q$ divides $p$.
\end{co}

\begin{proof} This is immediate from Theorem \ref{alphthm:A}.\end{proof}

\begin{re} If ${\rm rad}(p)$ strictly divides itself, the implication $(ii)\Rightarrow (i)$ is, in general, false. Let $p_{i}$ be the increasing enumeration of the primes, and let $p$ be generated by $\eta=p_{1}^{\alpha} \cdot p_{2}^{\alpha-1}\cdot\ldots\cdot p_{\alpha}$ for some infinite $\alpha$; clearly, ${\rm rad}(p)$ strictly divides itself. Notice that $\eta=f(\alpha)$ for $f$ a standard function; moreover, $p$ is not highly exponential (in fact, quite the opposite is true, as $PB(p)=p$). However, $p$ is maximal with respect to the divisibility. In fact, pick any $q\in\beta\N$ and let $\beta$ be a generator of $q$. Pick $\gamma\sim \alpha$ with $(\beta,\gamma)$ tensor. Then it is immediate to see that $\beta|p_{1}^{\gamma}\cdot\ldots\cdot p_{\gamma}=f(\gamma)\sim f(\alpha)$.
\end{re}

\section{Two exponentiations}\label{sec:twoexp}

In this section, we focus on the functions $E_{1}(p,q), E_{2}(p,q)$ introduced in \cite{LB} and recalled here in Definition \ref{def:exp}. Notice that, by \cite[Theorem 4.1]{HS}, for all $q\in\beta\N$ the functions ${p}\mapsto E_2({p},{q})$ and ${p}\mapsto E_1({p},{q})$ are continuous. When $p$ or $q$ are principal, it is immediate to notice that $E_{1}(p,q)=E_{2}(p,q)$; therefore, for $n\in\N$ we will write $p^n$ instead of $E_1(p,n)=E_2(n,p)$ and $n^p$ instead of $E_1(n,p)=E_2(p,n)$.

\begin{re}
It follows directly from the definition that $E_1({p},{q})$ and $E_2({p},{q})$ respectively are images of the tensor product ${p}\otimes{q}$ via extensions to $\beta(\mathbb{N}^2)$ of functions $(n,m)\mapsto n^m$ and $(n,m)\mapsto m^n$. Hence, for $(\alpha,\beta)\in\mu(p\otimes q)$ we have that $\alpha^{\beta}\in \mu(E_{1}(p,q))$, whilst $\beta^{\alpha}\in\mu (E_{2}(p,q))$.
\end{re}

In \cite{LB}, several properties of $E_{1},E_{2}$ were studied, and many algebraic questions about solvability of equations involving $E_{1},E_{2}$ were asked. We rephrase here one of those questions, which was the origin of our research.

\begin{pb}\label{FundProb}
Characterize nonprincipal ultrafilters $p,q$ for which 
\begin{equation}\label{FundEq} E_1({p},{q})=E_2({q},{p}). \end{equation} 
Equivalently, characterize tensor pairs $(\alpha_{1},\beta_{1}),(\beta_{2},\alpha_{2})$ in $^{\ast}\N^{2}$ for which $\alpha_{1}\sim\alpha_{2},\beta_{1}\sim\beta_{2},\alpha_{1}^{\beta_{1}}\sim\alpha_{2}^{\beta_{2}}$.
\end{pb}

\medskip

{\bfseries Convention:} From now on, in this section $p,q$ will denote nonprincipal ultrafilters, and all nonstandard numbers we write will be tacitly assumed to be in $^{\ast}\N\setminus\N$.

\medskip

As proven in \cite{LB}, under this additional request the desired equality (\ref{FundEq}) might hold (see \cite[Example 2.15]{LB}) or it might fail (see \cite[Example 2.11]{LB}). Let us begin our study of Problem $\ref{FundProb}$ by providing a very large class of examples of $E_{1}(p,q)\neq E_{2}(q,p)$.

\begin{ex}
Let $p,q$ be nonprincipal ultrafilters, with $p$ containing the complement of the set of perfect powers. Let $(\alpha_{1},\beta_{1})\in \mu(p\otimes q), (\beta_{2},\alpha_{2})\in\mu(q\otimes p)$. As $\alpha_{1},\alpha_{2}$ are not perfect powers, we have that $\alpha_{1}=PB(\alpha_{1}^{\beta_{1}})$, $\alpha_{2}=PB(\alpha_{2}^{\beta_{2}})$. If 
\[X=\{n\in\N\mid PB(n)<PE(n)\},\]
we have that $\alpha_{1}^{\beta_{1}}\in\,^{\ast}X$, whilst $\alpha_{2}^{\beta_{2}}\in\,^{\ast}X^{c}$. Hence $E_{1}(p,q)\neq E_{2}(q,p)$.
\end{ex}

By the above example, solutions of $E_{1}(p,q)=E_{2}(q,p)$ force the set of perfect powers to belong to $p$; e.g., any ultrafilter $p$ containg the set of primes satisfies $E_{1}(p,q)\neq E_{2}(q,p)$ for all $q\in\beta\N$.

In the only known solutions of Equation (\ref{FundEq}), $p$ has the form $p=n^{p_{1}}$ with $n\in\N_{\geq 2}$ and $p_{1}\in\beta\N\setminus\N$ such that $p_{1}\odot q=q\odot p_{1}$ (for details, see \cite[Example 2.15]{LB}). Our Theorem \ref{alphthm:B} proves that arbitrary solutions of Equation (\ref{FundEq}) must satisfy similar strict restrinctions.

\begin{proof}[Proof of Theorem \ref{alphthm:B}] Let $\alpha$ be a generator of $p$. Let $\beta_{1}<\beta_{2}$ be generators of $q$ with $(\beta_{1},\alpha),(\alpha,\beta_{2})$ tensor pairs, so that $\alpha^{\beta_{2}}$ generates $E_1({p},{q})$ and $\alpha^{\beta_{1}}$ generates $E_{2}(q,p)$. Notice that $f(\alpha^{\beta_{1}})=f(\alpha)=f(\alpha^{\beta_{2}})$, whilst $e_f(\alpha^{\beta_1})=e_f(\alpha)\beta_1$, $e_{f}(\alpha^{\beta_{2}})=e_{f}(\alpha)\beta_{2}$. As $E_1({p},{q})=E_2({q},{p})$, we have that $\alpha^{\beta_{1}}\sim\alpha^{\beta_{2}}$. Applying the function $(f,e_{f})$ to $\alpha^{\beta_{1}}$ and $\alpha^{\beta_{2}}$, we deduce by Lemma \ref{lem:list}.\ref{T1} that 
\begin{equation}\label{eq: paireq}\tag{$\diamond$} (f(\alpha),e_{f}(\alpha)\beta_{1})\sim (f(\alpha),e_{f}(\alpha)\beta_{2}).\end{equation}
(i) By Lemma \ref{lem:list}.\ref{T10}, $(f(\alpha),\beta_{2})$ is a tensor pair, hence $f(\alpha)\ll \beta_{2}$. By Lemma \ref{lem:list}.\ref{T12}, it follows that $f(\alpha)\ll e_{f}(\alpha)\beta_{2}$. By (\ref{eq: paireq}), it follows that $f(\alpha)\ll e_{f}(\alpha)\beta_{1}$. But $(\beta_{1},f(\alpha))$ is a tensor pair by Lemma \ref{lem:list}.\ref{T10}, hence $f(\alpha)\not\ll \beta_{1}$. Again from Lemma \ref{lem:list}.\ref{T12} we deduce hence that $f(\alpha)\ll e_{f}(\alpha)$. 

(ii) By (\ref{eq: paireq}), $e_{f}(\alpha)\beta_{1}\sim e_{f}(\alpha)\beta_{2}$. We conclude as, by Lemma \ref{lem:list}.\ref{T10}, $(e_{f}(\alpha),\beta_{2}),(\beta_{1},e_{f}(\alpha))$ are tensor pairs, hence $e_{f}(\alpha)\beta_{2}$ generates $e_{f}(p)\odot{q}$ whilst $\beta_{1}e_{f}(\alpha)$ generates ${q}\odot e_{f}({p})$. \end{proof}

\begin{ex} Assume that $E_1({p},{q})=E_2({q},{p})$. By Theorem \ref{alphthm:B}, $p$ is $f$-highly exponential for all factor-picking $f$. Let $\alpha=\prod_{i=1}^{\nu}p_{l(i)}^{k_{i}}$ be a generator of $p$, where $p_{l(i)}$ is the $l(i)$-th prime in $^{\ast}\N$ (hence $k_{i}>0$ for all $i\leq \nu$). In particular:
\begin{enumerate}
\item $PB(\alpha)\ll PE(\alpha)$, hence by Lemma \ref{highexp} $p_{l(i)}\ll k_{i}$ for all $i\leq \nu$;
\item if $\nu=m\in\N$, by taking for $j\leq m$ the factor-picking $f_{j}$ that maps $n$ to its $j$-th prime divisor, we deduce that $e_{f_{j}}(p)\odot q=q\odot e_{f_{j}}(p)$, namely $q$ commutes with all the ultrafilters generated by exponents of prime factorization of a generator of $p$.
\end{enumerate}
\end{ex}

The necessary conditions in Theorem \ref{alphthm:B} are independent, and neither of them suffices to prove that $E_1({p},{q})=E_2({q},{p})$, as shown in the following example.

\begin{ex} Let us show that $(1), (2)$ from Theorem \ref{alphthm:B} are independent.

\begin{enumerate}
\item It is known (see \cite[Theorem 6.10]{HS}) that the center of $(\beta\N,\odot)$ is $\N$. Pick $q$ nonprincipal, let $u$ be such that $u\odot q\neq q\odot u$ and let $p=2^{u}$. Then $p$ is $f$-highly exponential for all factor-picking $f$ with $f(2)\neq 1$, as in this case necessarily $f(2)=2$, whilst $e_{f}(p)=u$. Hence $e_{f}(p)\odot q=u\odot q\neq q\odot u=q\odot e_{f}(p)$ (and, in fact, $E_{1}(p,q)=2^{u\odot q}\neq 2^{q\odot u}= E_{2}(q,p)$).
\item Now let $p$ be generated by an infinite prime $\alpha$. Trivially, $e_{f}(\alpha)=1$ for all $f$ factor-picking, hence $e_{f}(p)\odot q=q\odot e_{f}(p)$ for all $q\in\beta\N$. But $p$ is not $f$-highly exponential, hence $E_{1}(p,q)\neq E_{2}(q,p)$ for all $q\in\beta\N\setminus\N$.
\end{enumerate}
\end{ex}

We do not know if the implication in Theorem \ref{alphthm:B} can be reversed. We will discuss this in Section \ref{sec:conclusions}.

\section{General exponentiation of ultrafilters}\label{sec:genexp}

Exponentiation of ultrafilters can be seen from a more general point of view\footnote{The same would be true for several other operations, but it is outside the scope of the present paper.}. In fact, let $E:\N^{2}\rightarrow \N$ be the function mapping $(x,y)$ into $x^{y}$. Then $E_{1},E_{2}$ are just the restriction of $E$ to, respectively, tensor products and ``switched'' tensor products, namely ultrafilters of the form $s(p\otimes q)$ where $s:\N^{2}\rightarrow \N^{2}$ maps $(x,y)$ into $(y,x)$. 

Therefore, Problem \ref{FundProb} can be generalized to the following.

\begin{pb} Characterize for which ultrafilters $p,q\in\beta\N^{2}$ the equality $E(p)=E(q)$ holds.\end{pb} 

In terms of generators, the problem above can be equivalently rephrased as follows:

\begin{pb}\label{pbcharpair} Characterize pairs $(\alpha_{1},\beta_{1}),(\alpha_{2},\beta_{2})\in\,^{\ast}\N^{2}$ such that $\alpha_{1}^{\beta_{1}}\sim\alpha_{2}^{\beta_{2}}$. \end{pb}

\begin{re}\label{ReformulationRm} By Lemma \ref{lem:list}.\ref{T2}, $\alpha_{1}^{\beta_{1}}\sim\alpha_{2}^{\beta_{2}}$ if and only if there exists $(\alpha_{3},\beta_{3})\sim (\alpha_{2},\beta_{2})$ such that $\alpha_{1}^{\beta_{1}}=\alpha_{3}^{\beta_{3}}$, which is easily settled: there must be some $\gamma\in\,^{\ast}\mathbb{Q}$ such that either
\begin{enumerate}
\item $\alpha_{3}=\alpha_{1}^{\gamma}, \beta_{1}=\beta_{3}\gamma$, or
\item $\alpha_{1}=\alpha_{3}^{\gamma}, \beta_{3}=\beta_{1}\gamma$.
\end{enumerate}
\end{re}

The problem is that, in general, it is not trivial to see if any of the above conditions hold true given the ultrafilters generated by $(\alpha_{1},\beta_{1}),(\alpha_{2},\beta_{2})$. The main problems are that $E(\alpha,\beta)$ can be written as an exponential in many different ways, and that it might not even be possible to recover $\alpha,\beta$ from $E(\alpha,\beta)$.

\begin{ex} Let $\eta\notin er(\mu)$, with $\mu>\eta$. Let $\alpha=\mu^{\eta},\beta=\frac{\mu!}{\eta}$. Then $\alpha\notin er(\alpha^{\beta})$: in fact, $\alpha^{\beta}=\mu^{\mu!}$. As $\mu^{\mu!}$ is clearly a function of $\mu$, $\alpha\in er(\mu^{\mu!})$ would entail $\alpha\in er(\mu)$, namely $\mu^{\eta}\in er(\mu)$, from which clearly $\eta\in er(\mu)$. Notice that $\alpha^{\beta}$ can be written as an exponential in uncountably many ways.

Let us notice, however, that $\alpha\in er(\alpha^{\beta})$ is possible. This happens, e.g., when $\alpha$ is not a perfect power, in which case $\alpha=PB(\alpha^{\beta})$ for all $\beta$. It also happens whenever $f:\mathbb{N}\rightarrow\mathbb{N}$ is such that there is some $k\in\mathbb{N}$ for which $|\{m\in\mathbb{N}:m^{f(m)}=n\}|\leq k$ for all $n\in\mathbb{N}$. In such a case, by transfer we have that for all $\alpha\in\,^{\ast}\N$ $|\{\beta\in\,^{\ast}{\mathbb{N}}:\beta^{f(\beta)}=\alpha\}|\leq k$, hence if we let $g_i:\N\rightarrow\N$ be the function mapping every $n\in\N$ to the $i$-th smallest number $m$ such that $m^{f(m)}=n$, then every $\alpha\in\,^{\ast}{\mathbb{N}}$ is equal to $\,^{\ast} g_i(\alpha^{\,^{\ast} f(\alpha)})$ for some $i$, and so $\alpha\in{\rm er}(\alpha^{\,^{\ast} f(\alpha)})$. This case includes all polynomial functions $f:\mathbb{N}\rightarrow\mathbb{N}$ with coefficients in $\mathbb{Z}$.
\end{ex}

However, as soon as one between $\alpha,\beta$ belongs to $er\left(\alpha^{\beta}\right)$, the other does too.

\begin{lm} Let $\alpha,\beta\in\,^{\ast}\N\setminus\N$. Then $\alpha\in er\left(\alpha^{\beta}\right)$ if and only if $\beta\in er\left(\alpha^{\beta}\right)$.
\end{lm}

\begin{proof} If $\alpha=f(\alpha^{\beta})$, then the set $A=\{n\in\N:(\exists m\in\N)f(n)^m=n\}$ belongs to the ultrafilter generated by $\alpha^{\beta}$. Let $g:\N\rightarrow \N$ be the function mapping $n$ to $1$ if $n\notin A$, and to the unique $m$ such that $n=f(n)^{m}$ if $n\in A$. Then, trivially, $\beta=g(\alpha^\beta)$. The converse is similar.
\end{proof}

As expectable, a case where we have a complete characterization in Problem \ref{pbcharpair} is when some $\alpha_{i}$ or $\beta_{i}$ is standard. In fact, since $(n,\alpha)$ and $(\alpha,n)$ are both tensors whenever $n\in\N$, the following basic properties are consequences of Remark \ref{ReformulationRm}:

\begin{enumerate}
\item $E(n,\alpha)=E(m,\beta)$ if and only if there exists $h\in\mathbb{Q}$ such that $n=m^{h}$ and $\beta=h\alpha$; in particular, $E(n,\alpha)\neq E(m,\alpha)$ for all $n\neq m$;
\item $E(\alpha,n)=E(\beta,m)$ if and only if there exists $h\in\mathbb{Q}$ such that $n=mh$ and $\alpha=\beta^{h}$; in particular, $E(\alpha,n)\neq E(\alpha,m)$ for all $n\neq m$.
\end{enumerate}

More in general, by Remark \ref{ReformulationRm} one can easily deduce characterizations for the cases $E(n,\beta_{1})=E(\alpha_{2},\beta_{2})$ and $E(\alpha_{1},n)=E(\alpha_{2},\beta_{2})$.

\medskip

{\bfseries Convention:} From now on, in this section, whenever we write $\alpha,\beta$ we mean them to be infinite nonstandard natural numbers.

\medskip

As said, in general, given arbitrary $\alpha,\beta$ it is difficult to characterize $E(\alpha,\beta)$ (and its preimage). By Remark \ref{ReformulationRm}, given $\gamma\in\,^{\ast}\N$, the problem of finding all possible pairs $(\alpha,\beta)$ with $E(\alpha,\beta)=\gamma$ can be reduced to the problem of finding all possible $\beta$ such that $\gamma\sim PB(\gamma)^{\beta}$. Therefore, we now reduce to the following problem:

\begin{pb} Let $\alpha$ be a non perfect power. Characterize the pairs $\beta_{1},\beta_{2}$ such that $\alpha^{\beta_{1}}\sim \alpha^{\beta_{2}}$. \end{pb}

Lemma \ref{FundRemCut} tells us that the only case where $\alpha^{\beta}\sim\alpha^{\beta^{\prime}}$ entails $\beta=\beta^{\prime}$ is when $\beta=f(\alpha)$ for some $f:\N\rightarrow\N$. Ultrafilters of this form have the nice property of working as ``cancellative bases'' for exponentiation, in the following sense.

\begin{te}\label{pfree} Let $p$ be generated by $\alpha^{f(\alpha)}$ for some non perfect power $\alpha$. For $i=1,2$ let $\pi_{i}:\N^{2}\rightarrow \N$ be the projection on the $i$-th coordinate. Then, for all distinct $q_{1},q_{2}\in\beta\N^2$ such that $\pi_1(q_1)=\pi_1(q_2)=p$ we have $E(q_{1})\neq E(q_{2})$.
\end{te}

\begin{proof} Assume the opposite, that $E(q_{1})=E(q_{2})$. Let $(\alpha_{i},\beta_{i})\in\mu (q_{i})$ for $i=1,2$ be such that $\alpha_1^{\beta_1}=\alpha_2^{\beta_2}$. As $p$ is generated by $\alpha^{f(\alpha)}$ for some $\alpha\in\,^{\ast}\N$, it follows that there are non perfect powers $\alpha_{3}\sim\alpha_{4}\in\,^{\ast}\N$ so that $\alpha_{i}=\alpha_{i+2}^{f(\alpha_{i+2})}$ for $i=1,2$. Hence $\alpha_3^{f(\alpha_{3})\beta_1}=\alpha_4^{f(\alpha_{4})\beta_2}$ which, as $\alpha_{3}$ and $\alpha_{4}$ are not perfect powers, forces $\alpha_{3}=\alpha_{4}$, hence $\beta_{1}=\beta_{2}$. Therefore $q_{1}=q_{2}$.
\end{proof}

\section{Conclusions}\label{sec:conclusions}

Our results leave us with a few open questions. First, our Theorem \ref{alphthm:A} serves as a bridge between exponentiation and divisibility; this link, however, might be further studied. The paper \cite{So8} considers the equivalence classes of powers of prime ultrafilters with respect to mutual divisibility; this can be generalized to arbitrary ultrafilters, thus replacing the question whether $E(p)=E(q)$ with the question whether these two ultrafilters are in the same equivalence class.

Turning to $E_{1}(p,q)=E_{2}(p,q)$, our main question is:

\begin{pb}
Is the conjunction of the conditions (i) and (ii) in Theorem \ref{alphthm:B} sufficient for $E_1({p},{q})=E_2(q,p)$?
\end{pb}

We are unsure what the answer to the previous question will be. In the case of a negative one, it would be interesting to understand what other necessary conditions can be imposed on pairs $(p,q)$ witnessing $E_{1}(p,q)=E_{2}(p,q)$. In the case of a positive one, it would be interesting to understand how to adapt the result to the more general equation $E(p)=E(q)$ for $p,q\in\beta(\N^{2})$. Related to this problem, we ask our final question, which is basically asking if Theorem \ref{pfree} can be reversed.

\begin{pb} Let $p\in\beta\N$ and $T_{p}=\{q\in\beta(\N^{2})\mid \pi_{1}(q)=\pi_{2}(q)=p\}$. Assume that for all distinct $q_{1},q_{2}\in T_{p}$ holds $E(q_{1})\neq E(q_{2})$. Is it true that generators of $p$ must be of the form $\alpha^{f(\alpha)}$ for some non perfect power $\alpha$?\end{pb}

{\it Acknowledgements.} The second author gratefully acknowledges financial support of the Ministry of Science, Technological Development and Innovation of the Republic of Serbia (Grants No.\ 451-03-33/2026-03/ 200125 and 451-03-34/2026-03/ 200125).

\end{document}